\documentclass[11pt,reqno]{amsart}
\usepackage[letterpaper, margin=1in]{geometry}
\usepackage{graphicx}
\usepackage{amsfonts}
\usepackage{amsmath}
\usepackage{amssymb}
\usepackage{mathtools}
\usepackage{hyperref}
\usepackage{tikz}
\usetikzlibrary{decorations.pathreplacing}
\usepackage{nicematrix}
\usepackage{graphics}
\usepackage[misc]{ifsym}
\usepackage{bbding}
\usepackage{todonotes}
\usepackage{blkarray}
\usepackage{paralist}
\usepackage[normalem]{ulem}

\newcommand{\zz}{\mathbb{Z}}
\newcommand{\nn}{\mathbb{N}}
\newcommand{\pp}{\mathbb{P}}
\newcommand{\rr}{\mathbb{R}}
\newcommand{\cc}{\mathbb{C}}

\newcommand{\A}{\mathcal{A}}
\newcommand{\B}{\mathcal{B}}
\newcommand{\C}{\mathcal{C}}
\newcommand{\D}{\mathcal{D}}
\newcommand{\M}{\mathcal{M}}

\newcommand{\sqbracket}[1]{\left[#1\right]}
\newcommand{\ti}{\tilde{i}}
\newcommand{\tj}{\tilde{j}}
\newcommand{\tk}{\tilde{k}}

\newcommand{\tbeta}{\tilde{\beta}}

\newcommand{\m}{\mathbf{m}}

\newcommand{\tfp}{\B\times_{\A}\C}
\newcommand{\xx}{\mathbf{x}}
\newcommand{\yy}{\mathbf{y}}

\newcommand{\fata}{\mathbf{a}}
\newcommand{\bb}{\mathbf{b}}
\newcommand{\cea}{\mathbf{c}}
\newcommand{\tea}{\mathbf{t}}
\newcommand{\fatp}{\mathbf{p}}
\newcommand{\fatq}{\mathbf{q}}

\DeclareMathOperator{\conv}{Conv}

\DeclareMathOperator{\relint}{Relint}

\newcommand{\minus}{\scalebox{0.75}[1.0]{$-$}}

\newtheorem{theorem}{Theorem}[section]
\newtheorem{lemma}[theorem]{Lemma}
\newtheorem{proposition}[theorem]{Proposition}
\newtheorem{corollary}[theorem]{Corollary}

\theoremstyle{definition}
\newtheorem{definition}[theorem]{Definition}
\newtheorem{example}[theorem]{Example}

\theoremstyle{remark}
\newtheorem{remark}[theorem]{Remark}

\numberwithin{equation}{section}

\begin{document}

\title{Toric Fiber Products in Geometric Modeling}

\author{Eliana Duarte}

\author{Benjamin Hollering}

\author{Maximilian Wiesmann}

\subjclass[2020]{62R01, 52B20, 13P25, 14M25}

\keywords{Toric variety, Exponential family, Blending function, Maximum Likelihood Estimation, Linear Precision, Toric Fiber Product, Horn Parametrization}

\begin{abstract}
An important challenge in Geometric Modeling is to classify polytopes with rational linear precision. Equivalently, in
Algebraic Statistics one is interested in classifying scaled toric varieties, also known as discrete exponential families, for which the maximum likelihood estimator  can be written in closed form as a rational function of the data (rational MLE). The toric fiber product (TFP) of statistical models is an operation to iteratively construct new models with rational MLE from lower dimensional ones. In this paper we introduce TFPs to the Geometric Modeling setting to construct polytopes with rational linear precision and give explicit formulae for their blending functions. A special case of the TFP is taking the Cartesian product of two polytopes and their blending functions.
 The Horn matrix of a statistical model with rational MLE is a key player in both Geometric Modeling and Algebraic Statistics; it proved to be fruitful providing a characterization of those polytopes having the more restrictive property of strict linear precision. We give an explicit description of the Horn matrix of a TFP.
\end{abstract}

\maketitle

\section{Introduction}
A discrete statistical model with $m$ outcomes is a subset $\M$ of the open probability simplex $\Delta_{m-1}^{\circ}=\{(p_1,\ldots, p_m): p_i> 0, \sum p_i=1\}$. 
Each point in $\Delta_{m-1}^\circ$ specifies a probability 
distribution for a random variable $X$ with outcome space
$[m]:=\{1,\ldots,m\}$ by setting $p_i=P(X=i)$. Given an i.i.d.\ 
 sample $\D=\{X_1,\ldots,X_{N}\}$ of $X$,
let $u_i$ be the number of times the outcome
$i$ appears in $\D$ and set $u=(u_1,\ldots,u_m)$.
The maximum likelihood estimator of the model
$\M$ is the function $\Phi:\nn^{m}\to \M$ that assigns
to $u$ the point in $\M$ that maximizes the log-likelihood
function $\ell(u|p):=\sum_{i}u_i\log(p_i)$.
For  discrete regular exponential families,
the log-likelihood function is concave, and under
certain genericity conditions on $u\in \nn^{m}$, existence and uniqueness of the maximum likelihood estimate $\Phi(u)$ is guaranteed \cite{haberman74}. This does not mean that the MLE is given in closed form  but rather that it can be computed using iterative proportional scaling \cite{Darroch-Ratcliff:1972}.

In Algebraic Statistics, discrete exponential families
are studied from an algebro-geometric perspective
using the fact that the Zariski closure of  any such family is a scaled projective toric variety, we refer to these as toric varieties from this point forward.  In this setting, the
complexity of  maximum likelihood estimation for a model $\M$, or more generally any algebraic variety, is measured in terms of its maximum likelihood degree (ML degree). The ML degree of $\M$ is the number of critical points 
of the likelihood function over the complex numbers for generic $u$ and it is an invariant of $\M$ \cite{huh-sturmfels-2014}. If a model
has ML degree one it means that the coordinate functions of $\Phi$ are rational functions in $u$, thus
the MLE has a closed form expression which is in fact determined
completely in terms of a Horn matrix as explained in \cite{Huh14,DMS21}. It is an open problem in Algebraic Statistics to characterize the class of toric varieties with ML degree one and their respective Horn matrices.

The toric fiber product (TFP), introduced by Sullivant \cite{sullivant2007toric}, is an operation that takes
two toric varieties $\M_1,\M_2$ and, using compatibility criteria determined by a multigrading
$\A$, creates
a higher dimensional toric variety $\M_1\times_\A \M_2$. This operation is used to construct a Markov basis for  $\M_1\times_\A \M_2$  by using  Markov bases of $\M_1$ and $\M_2$. Interestingly, the ML degree of a TFP
is the product of the ML degrees of its factors, therefore the TFP
of two models with ML degree one yields  a model with ML degree one \cite{amendola2020maximum}. 
The Cartesian product of two statistical models is an instance of a TFP. Another  example is the class of decomposable graphical models, each of these models has ML degree one and can be constructed iteratively from lower dimensional ones using TFPs \cite{sullivant2007toric,L96}.

In Geometric Modeling, it is an open problem to classify 
polytopes in dimension $d\geq 3$ having rational linear precision \cite{clarke2020moment}. Remarkably, a polytope has rational linear precision
if and only if its corresponding toric variety has ML degree one \cite{garcia-puente2010}. Inspired by Algebraic Statistics, it is our goal in this article to
introduce the toric fiber product construction to
 Geometric Modeling. In statistics, the interest is
in the closed form expression for the MLE; in Geometric Modeling,
the  interest is in explicitly writing blending functions defined on the polytope that satisfy the property of linear precision. Our main Theorem~\ref{thm} gives an explicit formula for
the blending functions defined on the toric fiber product of two polytopes that have rational linear precision.

For certain toric varieties with ML degree one,  geometric
information about their associated polytopes determines a Horn matrix for the model.
Instances of this phenomena are present in the characterization of polytopes with the more restrictive property of strict linear precision \cite{clarke2020moment}, and in the classification of 2D toric models with ML degree one \cite{familiesEliana}.
With the aim to facilitate the study of these ideas in future work, 
we provide, in Section~\ref{sec:Horn}, an explicit construction of 
a Horn matrix for the  toric fiber product of two toric varieties with ML degree one. This construction
reformulates \cite[Thm.\ 5.5]{amendola2020maximum} in terms of Horn matrices.

\section{Preliminaries}
In this section we provide background on blending functions, rational linear precision, scaled projective toric varieties and toric fiber products. For a friendly introduction to Algebraic Statistics, we refer the reader to the book by Sullivant \cite{S19},  in particular to Chapter 7 on maximum likelihood estimation. To the readers looking for more background on toric geometry we recommend the book by Cox, Little and Schenck \cite{cls11}.

\subsection{Blending Functions} \label{sec:bledningfunctions}
Let $P\subset \rr^d$ be a lattice polytope with facet representation
$P=\{\fatp\in \rr^d:
\langle \fatp,n_i \rangle \geq a_i, \forall i\in [R]\}$,  where $n_i$ is a primitive inward facing normal vector to the facet $F_i$. Without loss of generality, we will always assume that $P$ is full-dimensional inside $\rr^d$. The lattice distance of a point $\fatp\in \rr^d$ to 
$F_i$  is
$h_{i}(\fatp):=\langle \fatp,n_i\rangle+a_i$,   $i \in [R]$.
Set $\B:=P\cap \zz^d$, so $\B$ is the set of lattice points in $P$ and let $w=(w_{\bb})_{\bb\in \B}$ be a vector of positive weights. To each $\bb\in \B$ we associate the rational functions
 $\beta_{\mathbf{b}},\beta_{w},\beta_{w,\bb}:P\to \rr$ defined by
 \begin{align}
    \beta_{\mathbf{b}}(\fatp)&:=\prod_{i=1}^{R} h_{i}(\fatp)^{h_{i}(\mathbf{b})}, &\beta_{w}(\fatp)&:=\sum_{\bb\in\B}w_{\bb}\beta_{\bb}(\fatp), \text{ and}  &\beta_{w,\mathbf{b}}&:=w_{\mathbf{b}}\beta_{\bb}/\beta_{w} .
 \end{align}
The functions $\beta_{w,\bb}$, $\bb\in \B$, are the \emph{toric blending functions} of the pair $(P,w)$, introduced by Krasauskas \cite{krasauskas2002} as generalizations of B\'ezier curves and surfaces to more general polytopes. Blending functions usually satisfy additional properties that make them amenable for computation, see for instance \cite{krasauskas2002}.  Given a set of control points $\{Q_{\bb}\}_{\bb\in \B}$, a \emph{toric patch} is defined by the rule $F(\fatp):=\sum_{\bb\in \B}\beta_{\bb}(\fatp)Q_{\bb}$.

The \emph{scaled 
projective toric variety} $X_{\B,w}$ is the Zariski closure
of the image of the map $(\cc^*)^{d}\to \pp^{|\B|-1}$ defined 
by $\tea \mapsto [w_{\bb}\tea^\bb]_{\bb\in \B}$.  Here $\tea=(t_1,\ldots, t_d), \bb=(b_1,\ldots,b_d)$ and $\tea^\bb=\prod_{i\in[d]}t_i^{b_i}$. The image 
of $X_{\B,w}$ under the map $\pp^{|\B|-1}\to \cc^{|\B|}$, $[x_1:\cdots:x_{|\B|}]\mapsto
\frac{1}{x_1+\cdots+x_{|\B|}}(x_1,\cdots,x_{|\B|})$ intersected
with the positive orthant defines a discrete regular exponential 
family $\M_{\B,w}$ inside $\Delta_{|\B|-1}^{\circ}$. In the literature these are also called log-linear models. 
In this construction we require that the vector of ones is in the
rowspan of the matrix whose columns are the points in $\B$. If this is not the case, we add the vector of ones to this matrix.

\begin{definition}
\label{def:rationalLinearPrecision}
The pair $(P,w)$ has \emph{rational linear precision} if there is a set of rational functions $\{\hat{\beta}_{\bb}\}_{\bb\in \B}$
on $\cc^d$ satisfying:
\begin{compactenum}
    \item $\sum_{\bb\in \B} \hat{\beta}_{\bb}=1$.
    \item The functions $\{\hat{\beta}_{\bb}\}_{\bb\in \B}$ define a rational parametrization
    \[\hat{\beta}:\cc^d\dashrightarrow X_{\B,w}\subset \pp^{|\B|-1},\quad \hat{\beta}(\tea)=(\hat{\beta}_{\bb}(\tea))_{\bb\in \B}.\]
    \item For every $\fatp\in \relint(P)\subset \cc^d$, $\hat{\beta}_{\bb}(\fatp)$ is defined and is a nonnegative real number.
    \item Linear precision: $\sum_{\bb\in \B}\hat{\beta}_{\bb}(\fatp)\bb=\fatp$ for all $\fatp\in P$.
\end{compactenum}
\end{definition}
The property of rational linear precision does not hold for arbitrary toric
patches but it is desirable because  the blending functions ``provide barycentric coordinates for general 
control point schemes'' \cite{garcia-puente2010}. A deep relation to Algebraic Statistics is provided by the following statement.
\begin{theorem}[\cite{garcia-puente2010}]
\label{thm:mlDegRationalPrecision}
The pair $(P,w)$ has rational linear precision if and only if  $X_{\B,w}$ has ML degree one.
\end{theorem} 

\begin{remark}
  Henceforth, to ease notation, 
  we drop the usage of a vector of weights $w$ for the blending
  functions $\beta_{w,\bb}$ and the scaled projective toric variety $X_{\B,w}$. Although we will not in general write them explicitly in the proofs, the weights
  play an important role in determining whether the toric variety
  has ML degree one or, equivalently, if the polytope has rational linear precision. A deep dive into the
  study of these scalings for toric varieties by using principal $A$-determinants is presented in \cite{amendola2019}.
\end{remark}
\begin{example} \label{ex:tbfs}
    Consider the point configurations  $\B=\{(0,0),(1,0),(0,1),(1,1)\}$, $\C=\{(0,0),(1,0),$ $(2,0),(1,1),(0,1)\}$ and set $P=\conv(\B)$, $Q=\conv(\C)$; these are displayed in Figure~\ref{fig:squareTrapezoidTFP} . The facet presentation of
    $P$ is
    \[
    P=\{(x_1,x_2) \in \rr^2: x_1\geq 0, x_2\geq 0, 1-x_1\geq 0, 1-x_2\geq 0\}.
    \]
    The lattice distance functions of a point $(x_1,x_2)\in \rr^2$
    to the facets of $P$ are
    \[
    h_1=x_1,~h_2=x_2,~h_3=1-x_1,~h_4=1-x_2.
    \]
    Therefore the toric bleding functions of $P $ with weights $w=(1,1,1,1)$
    are:
    \begin{align} \label{tbf:square}
    \beta_{\tiny\begin{pmatrix}0 \\ 0
    \end{pmatrix}}&= (1\minus x_1)(1\minus x_2),   &\!\!\beta_{\tiny\begin{pmatrix}1 \\ 0
    \end{pmatrix}}&=x_2(1\minus x_1),&\!\! \beta_{\tiny\begin{pmatrix}0 \\ 1
    \end{pmatrix}}&=x_1(1\minus x_2), &\!\! \beta_{\tiny\begin{pmatrix}1 \\ 1
    \end{pmatrix}}&=x_1x_2.
    \end{align}
     These toric blending functions satisfy the conditions in Definition~\ref{def:rationalLinearPrecision}; when this is the case, $P$
    is said to have \emph{strict linear precision}.
    The polytope $Q$ has rational linear precision for the vector of weights 
    $w=(1,2,1,1,1)$. In this
    case, the toric blending functions do not satisfy condition $4$ in Definition~\ref{def:rationalLinearPrecision}, however, as explained in
    \cite{clarke2020moment}, the following functions do:
    \begin{align*}
        \tbeta_{\tiny\begin{pmatrix} 0 \\ 0
        \end{pmatrix}} &= \frac{(1\minus y_2)(2\minus y_1\minus y_2)^2}{(2\minus y_2)^2}, 
        & \tbeta_{\tiny\begin{pmatrix} 1 \\ 0
        \end{pmatrix}} &= \frac{2y_1(1\minus y_2)(2\minus y_1\minus y_2)}{(2\minus y_2)^2}, &\tbeta_{\tiny\begin{pmatrix} 2 \\ 0
        \end{pmatrix}} &= \frac{y_1^2(1\minus y_2)}{(2\minus y_2)^2}, \\
        \tbeta_{\tiny\begin{pmatrix} 0 \\ 1
        \end{pmatrix}} &= \frac{y_2(2\minus y_1\minus y_2)}{2\minus y_2}, & \tbeta_{\tiny\begin{pmatrix} 1 \\ 1
        \end{pmatrix}} & = \frac{y_1 y_2}{2\minus y_2}. & &
    \end{align*}
    
\end{example}

\begin{figure}[t]
\centering
\includegraphics[width=\textwidth]{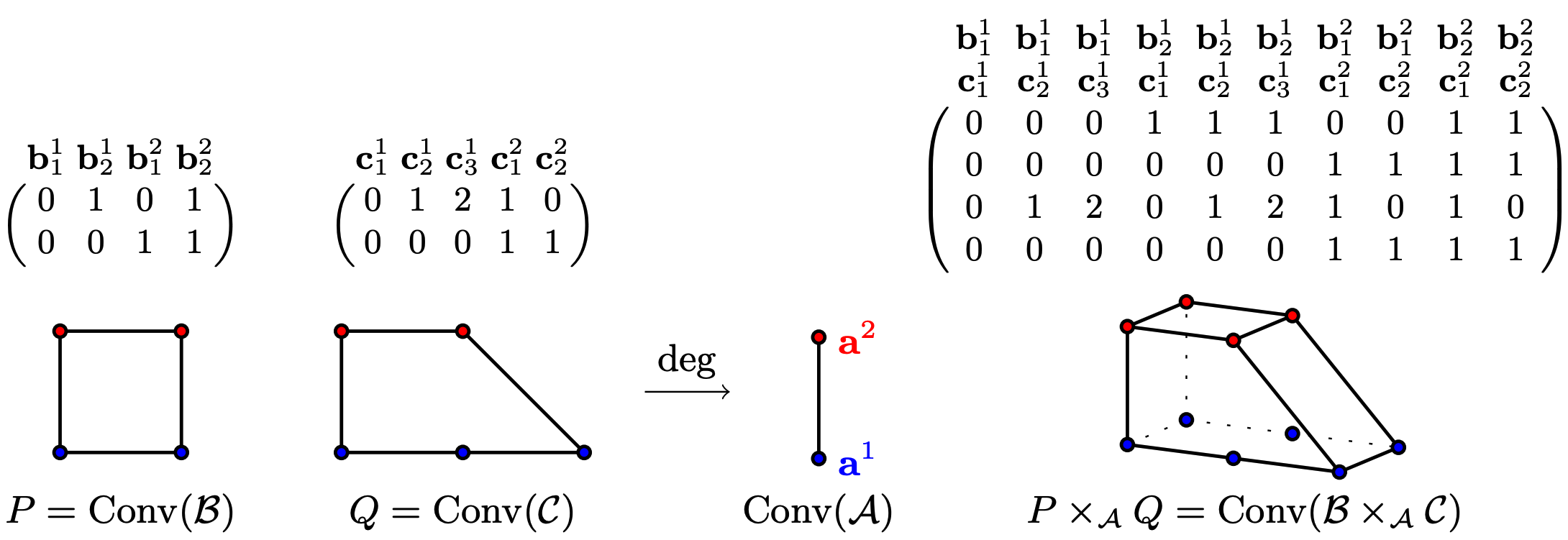}

\caption{Toric fiber product of the point configurations $\B$ and $\C$ in
Example~\ref{ex:squareTrapezoidTFP}. Each point configuration is displayed as a matrix with its corresponding convex hull below.
 The blue vertices in each polytope have degree $e_1$ while the red vertices in each polytope have degree $e_2$ in the associated multigrading $\A$. The degree map is $\deg(\bb_j^i)=\deg(\cea_{k}^i)=\fata^i$.}
\label{fig:squareTrapezoidTFP}
\end{figure}

\subsection{Toric Fiber Products of Point Configurations}
\label{sec:TFP}
Let $r\in\mathbb{N}$ and $s_i, t_i\in \mathbb{N}$ for $1\leq i\leq r$. Fix integral point configurations $\A = \{\fata^i : i\in\sqbracket{r}\}\subseteq \mathbb{Z}^d$, $\B = \{\bb^i_j : i\in\sqbracket{r}, j\in \sqbracket{s_i}\}\subseteq \mathbb{Z}^{d_1}$ and $\C=\{\cea^i_k : i\in\sqbracket{r},k\in\sqbracket{t_i}\}\subseteq \mathbb{Z}^{d_2}$. 
For any point configuration $\mathcal{P}$, we use $\mathcal{P}$  interchangeably to denote a set of points or the matrix whose columns are the points in $\mathcal{P}$; the symbol
 $|\mathcal{P}|$ will be used to denote the indexing set of $\mathcal{P}$.  For each $i\in |\A|$, set $\B^{i}:=\{\bb_{j}^{i}:j\in [s_i]\}$
and $\C^{i}=\{\cea_{k}^{i}:k\in [t_i]\}$. The indices $i,j,k$
are reserved for elements in $|\A|,|\B^i|$ and $|\C^i|$, respectively.

Throughout this paper, we assume linear independence of $\A$ and the existence of a $\overline{\omega}\in\mathbb{Q}^d$ such that $\overline{\omega} \fata^i = 1$ for all $i$; the latter condition ensures that if an ideal is homogeneous with respect to a multigrading in $\A$ it is also homogeneous in the usual sense. Sullivant introduces the TFP as an operation on toric ideals which are multigraded by $\A$; such condition, as explained in \cite{engstrom2014multigraded}, is equivalent to the existence of linear maps $\pi_1:\zz^{d_1}\to \zz^r$ and $\pi_{2}:\zz^{d_2}\to\zz^r$ such that $\pi_{1}(\bb_{j}^i)=\fata^{i}$
for all $i$ and $j$, and $\pi_{2}(\cea_{k}^{i})=\fata^i$ for all $i$ and $k$.
We use $\deg$ to denote the projections $\pi_1,\pi_2$.

The \emph{toric fiber product} of  $\B$ and $\C$  is the point configuration 
$\B\times_{\A}\C$ given by
\[\B\times_{\A}\C = \{(\bb^i_j, \cea^i_k) : i\in|\A|, j\in |\B^{i}|, k\in|\C^{i}|\}.\]
In terms of toric varieties, introduced in Section~\ref{sec:bledningfunctions}, the toric fiber product of $X_{\B}$ and $X_{\C}$ is the toric variety $X_{\B\times_{\A}\C}$ associated to $\B\times_{\A}\C$ which is given in the following way. Let $X_{\B}$ and $X_{\C}$ have coordinates $x_j^i$ and $y_k^i$ respectively. Then $X_\tfp = \phi(X_\B \times X_\C)$ where $\phi$ is the monomial map
\begin{align*}
\phi ~:~ \cc^{|\B|} \times \cc^{|\C|} &\to \cc^{|\tfp|} \\
      (x_j^i, y_k^i) &\mapsto x_j^iy_k^i = z_{jk}^i. 
\end{align*}    
Furthermore, if $w,\tilde{w}$ are weights for $\B,\C$, respectively,
then the vector of weights for $\tfp$ is $w_{\tfp}:=(w_j^i\tilde{w}_{k}^i)^{i\in|\A|}_{(j,k)\in|\B^i\times \C^i|}$.
We end this section with an example illustrating this operation. 
\begin{example}
\label{ex:squareTrapezoidTFP}
Consider the point configurations $\B$ and $\C$ in Example~\ref{ex:tbfs} and let $\A = \{e_1, e_2 \}$ consist of two standard basis vectors. The construction of
a degree map and the corresponding toric fiber product $\B\times_{\A}\C$ is explained in Figure~\ref{fig:squareTrapezoidTFP}.

\end{example}

\section{Blending Functions of Toric Fiber Products}
In this section we show that the blending functions of the toric fiber product of two polytopes with rational linear precision can be constructed from the blending functions of the original polytopes and give an explicit formula for them. 
Throughout this section we use the setup for the toric fiber product introduced in Section~\ref{sec:TFP}. We let $P=\conv(\B)$ and $Q=\conv(\C)$ be polytopes with rational linear precision and denote their blending functions satisfying Definition~\ref{def:rationalLinearPrecision} by $\{\beta_j^i\}_{j \in |\B_i|}^{i \in |\A|}$ and $\{\beta_k^i\}_{k \in |\C_i|}^{i \in |\A|}$, respectively. 

\begin{theorem}
\label{thm}
If $P$ and $Q$ are polytopes with rational linear precision for weights $w,\tilde{w}$, respectively, then the toric fiber product $P \times_\A Q$ has rational linear precision with vector of weights $w_{\tfp}$. Moreover, blending functions with rational linear precision for $P \times_\A Q$ are given by
\begin{equation}
    \label{equ:blendingFct}
    \beta_{j,k}^{i}(\fatp, \fatq) = 
    \frac{\beta_{j}^{i}(\fatp) \beta_{k}^{i}(\fatq)}{\sum_{j' \in |\B^i|} \beta^{i}_{j'}(\fatp)} = \frac{\beta_{j}^{i}(\fatp) \beta_{k}^{i}(\fatq)}{\sum_{k' \in |\C^i|} \beta^{i}_{k'}(\fatq)}
\end{equation}
where $(\fatp, \fatq) \in P \times_\A Q$. 
\end{theorem}

\begin{remark}
\label{rem:defBlendingFct}
   The two expressions on the right hand side of Equation (\ref{equ:blendingFct})
   are well defined on $\relint(P\times_{\A}Q)$. The morphism $\beta_{j,k}^i$ extends to a rational function $\beta^i_{j,k}\colon \cc^{d} \dashrightarrow \cc$ where $d=\dim(P\times_{\A}Q)$. By abuse of notation, we will sometimes write $\beta^i_{j,k}(\tea) = \frac{1}{N^i(\tea)}\beta^i_j(\tea)\beta^i_k(\tea)$ where $\tea\in\cc^d$ and $N^i(\tea)$ denotes the denominator as in (\ref{equ:blendingFct}). 
\end{remark}
The following example illustrates the construction in Theorem~\ref{thm}. 
\begin{example}
\label{ex:squareTrapezoidBlendingFunctions}
Consider the polytopes $P$ and $Q$ from Example~\ref{ex:tbfs},
with their vectors of weights. 
By Theorem \ref{thm},  the blending functions for $P \times_\A Q$ are $\frac{\beta_j^i \tbeta_k^i}{\sum_{j} \beta_j^i }= \frac{\beta_j^i \tbeta_k^i}{\sum_{k}\tbeta_k^i}$. For example, the blending function corresponding to $(\bb^1_2 ~ \cea^1_3)^T$ is
\[\beta^1_{2,3} = \frac{\beta^1_2 \tbeta_3^1}{\beta_1^1 + \beta_1^2} = \frac{x_1(1\minus x_2)y_1^2(1\minus y_2)}{(1\minus x_2)(2-y_2)^2}=\frac{\beta^1_2 \tbeta_3^1}{\tbeta_1^1 + \tbeta_2^1 + \tbeta_3^1} = \frac{x_1(1\minus x_2)y_1^2(1\minus y_2)}{(1\minus y_2)(2-y_2)^2}. \]
Note that while the denominators are not the same, the two expressions above are equal at all points in $\relint(P \times_\A Q)$.  
\end{example}

Before proving Theorem \ref{thm} we will prove two lemmas which will be used in the final proof. Our first lemma demonstrates how the blending functions behave on certain faces of $P$ and $Q$. The second lemma shows that the two parametrizations in Equation (\ref{equ:blendingFct}) yield the same MLE for a generic data point $u$. 

\begin{lemma}
\label{lem:subBlendingFctsSumTo1}
Let $P^i$ be the subpolytope defined by $P^i=\conv\{\bb^i_j : j\in|\B^i|\}$. Then, for $\fatp\in P^i$, we have
\[\sum_{j\in |\B^i|} \beta^i_j(\fatp) = 1.\]
\end{lemma}

\begin{proof}
By assumption, $\beta\colon \cc^{d_1}\dashrightarrow X_{\B},~\beta(\tea)=\left(\beta^i_j(\tea)\right)^{i\in|\A|}_{j\in|\B^i|}$ is a 
rational parametrization of $X_{\B}$. Let $X^i_{\B}$ be the toric variety associated to $P^i$; we claim that $X^i_{\B}$ is parametrized by $\left(\beta^i_j(\tea)\right)_{j\in|\B^i|}$ and setting all other coordinates of $\beta$ to zero. Indeed, consider the linear map
\[\deg\colon P\rightarrow \conv(\A),~ \bb^i_j\mapsto \fata^i.\]
As $\A$ is linearly independent, $\fata^i$ is a vertex of $\conv(\A)$. Note that $P^i=\deg^{-1}(\fata^i)$; as preimages of faces under linear maps are again faces, $P^i$ is a face of $P$. The claim then follows from the Orbit-Cone Correspondence \cite[Thm.\ 3.2.6]{cls11}. We know that $\sum_{(i,j)\in |\B|}\beta^i_j = 1$. On $P^i$, all $\beta^{i^{\prime}}_j$ for $i^{\prime}\neq i$ vanish, so we must have $\sum_{j\in|\B^i|}\beta^i_j(\fatp)=1$ for $\fatp\in P^i$.
\end{proof}

We record the following fact as a consequence from the proof above.
\begin{corollary} Let $P$ be a polytope equipped with a linearly independent multigrading $\A$. Then $P^i = \conv(\bb_j^i ~|~ j \in |\B^i|)$ is a face of $P$. 
\end{corollary}
\begin{example}
For the polytope $Q$ in Figure~\ref{fig:squareTrapezoidTFP}, we have
$Q^{1}=\conv(\cea_{1}^{1},\cea_{2}^{1},\cea_{3}^{1})$ and
$Q^2=\conv(\cea_1^2,\cea_2^2)$. The projection $\deg$ is illustrated in Figure~\ref{fig:squareTrapezoidTFP}.
To illustrate the result of Lemma~\ref{lem:subBlendingFctsSumTo1}, note that the sum of the blending
functions associated to the lattice points in $Q^1$ is equal to $1-y_2$.
\end{example}

\begin{lemma}
\label{lem:fctsAgreeOnMLE}
Let $P$ and $Q$ be polytopes with rational linear precision and  $\beta_1,\beta_2$ be two rational functions defined by
\[\beta_1(\tea)=\left(\frac{\beta_{j}^{i}(\tea) \beta_{k}^{i}(\tea)}{\sum_{j' \in |\B^i|} \beta^{i}_{j'}(\tea)}\right)^{i\in|\A|}_{(j,k)\in|\B^i\times \C^i|}
\text{, } \beta_2(\tea) = \left(\frac{\beta_{j}^{i}(\tea) \beta_{k}^{i}(\tea)}{\sum_{k' \in |\C^i|} \beta^{i}_{k'}(\tea)}\right)^{i\in|\A|}_{(j,k)\in|\B^i\times \C^i|}.\] 
For $u=\left(u^i_{j,k}\right)^{i\in|\A|}_{(j,k)\in|\B^i\times \C^i|}$, set
$\fatp = \sum_{(i,j,k)\in|\tfp|} \frac{u^i_{j,k}}{u^+_{+,+}}\m^i_{j,k}\in \cc^d.$
Then the maximum likelihood estimate for $X_{\tfp}$ is 
\[\beta_1(\fatp) = \beta_2(\fatp) = \left(\hat{p}^{i}_{j,k}\right)^{i\in|\A|}_{(j,k)\in|\B^i\times \C^i|}~.\] 
\end{lemma}

\begin{proof}
As $P$ and $Q$ have rational linear precision, by \cite[Prop.\ 8.4]{clarke2020moment} we have $\beta^i_j(\fatp) = \left(\hat{p}_{\B}\right)^i_{j}$ and $\beta^i_k(\fatp) = \left(\hat{p}_{\C}\right)^i_{k}$. Furthermore, by \cite[Thm.\ 5.5]{amendola2020maximum}, the MLE of the toric fiber product is given by $\hat{p}^i_{j,k} = \frac{\left(\hat{p}_{\B}\right)^i_j\left(\hat{p}_{\C}\right)^i_k}{\left(\hat{p}_{\A}\right)^i}$.
From the proof of \cite[Lem.\ 5.10]{amendola2020maximum}, as a consequence of Birch's Theorem, it follows that $\left(\hat{p}_{\B}\right)^i_+ = \frac{u^i_{+,+}}{u^+_{+,+}} = \left(\hat{p}_{\A}\right)^i$, and analogously $\left(\hat{p}_{\C}\right)^i_+ = \left(\hat{p}_{\A}\right)^i$. Therefore, 
\[\sum_{j' \in |\B^i|} \beta^{i}_{j'}(\fatp) = \left(\hat{p}_{\B}\right)^i_+ = \sum_{k' \in |\C^i|} \beta^{i}_{k'}(\fatp) = \left(\hat{p}_{\C}\right)^i_+ = \left(\hat{p}_{\A}\right)^i\]
and the desired statement follows.  
\end{proof}

We are now ready to prove Theorem~\ref{thm}.
\begin{proof}
Having rational linear precision is equivalent to having ML degree one by Theorem \ref{thm:mlDegRationalPrecision}. Then the first statement is a direct consequence of the multiplicativity of the ML degree under toric fiber products \cite[Thm.\ 5.5] {amendola2020maximum}.

We first show that both expressions in (\ref{equ:blendingFct}) define rational parametrizations
\[\beta_1(\tea)=\left(\frac{\beta_{j}^{i}(\tea) \beta_{k}^{i}(\tea)}{\sum_{j' \in |\B^i|} \beta^{i}_{j'}(\tea)}\right)^{i\in|\A|}_{(j,k)\in|\B^i\times \C^i|}
\text{, } \beta_2(\tea) = \left(\frac{\beta_{j}^{i}(\tea) \beta_{k}^{i}(\tea)}{\sum_{k' \in |\C^i|} \beta^{i}_{k'}(\tea)}\right)^{i\in|\A|}_{(j,k)\in|\B^i\times \C^i|}\] 
of $X_{\tfp}$.
To do this, we first show that the products $\beta_j^i \beta_k^i$ parametrize $X_{\tfp}$ and the result then follows since $\beta_1$ and $\beta_2$ are equivalent to $\beta_j^i \beta_k^i$ under the torus action associated to the multigrading $\A$. 
Let $\phi: \cc^{|\B|} \times \cc^{|\C|} \to \cc^{|\B \times_\A C|}$ be the map given by
\[
\phi(\xx, \yy) = (x_j^i y_k^i)_{(j,k) \in |\B^i\times \C^i|}^{i \in |\A|}~.
\] 
Then the toric fiber product $X_{\tfp}$ is precisely given by $\phi(X_\B \times X_\C)$. Since the blending functions $\beta_j^i$ and $\beta_k^i$ parametrize $X_\B$ and $X_\C$, respectively, and $\beta_j^i \beta_k^i = \phi \circ (\beta_j^i, \beta_k^i)$, we immediately get that the $\beta_j^i \beta_k^i$ parametrize $X_{\tfp}$. Now observe that the multigrading $\A$ induces an action of the torus $T_{\A} = (\cc^*)^{|\A|}$ via
\[T_{\A} \times X_{\tfp}\rightarrow X_{\tfp},~ (t^1,\dots,t^{|\A|}).\left(x^i_{j,k}\right)^{i\in|\A|}_{(j,k)\in|\B^i\times \C^i|} = \left(t^ix^i_{j,k}\right)^{i\in|\A|}_{(j,k)\in|\B^i\times \C^i|}.\]
Define $\tau:\cc^{d_1}\to T_{\A}$ by
\[\tau = (\tau^1,\dots,\tau^{|\A|}),~\tau^i(\tea) = \left\{\begin{array}{ll}
   \left(\sum_{j \in |\B^i|}\beta^i_j(\tea) \right)^{-1} &  \text{ if } \sum_{j \in |\B^i|}\beta^i_j(\tea) \neq 0\\
   1  & \text{ else.}
\end{array}\right. \]
Note that $\tau(\tea) \in T_{\A}$ and $\tau(\xx).(\beta_j^i(\xx) \beta_k^i(\xx)_{j \in |\B^i|, k \in |\C^i|}^{i \in |\A|} = \beta_1(\xx)$ for all $\xx\in P\times_{\A}Q$, showing that $\beta_j^i(\xx) \beta_k^i(\xx)$ and $\beta_1(\xx)$ lie in the same $T_{\A}$-orbit. A similar argument shows the same for $\beta_2(\xx)$, thus both $\beta_1$ and $\beta_2$ parametrize $X_{\tfp}$.

We will now show the two expressions in Equation \ref{equ:blendingFct} are equal. Let us define a new $\tau:\cc^{d_1+d_2}\to T_{\A}$ by
\[\tau = (\tau^1,\dots,\tau^{|\A|}),~\tau^i(\tea) = \left\{\begin{array}{ll}
   \frac{\sum_{j \in |\B^i|}\beta^i_j(\tea)}{\sum_{k \in |\C^i|}\beta^i_k(\tea)}  &  \text{ if } \sum_{j \in |\B^i|}\beta^i_j(\tea) \neq 0 \neq \sum_{k \in |\C^i|}\beta^i_k(\tea)\\
   1  & \text{ else.}
\end{array}\right.\]
Clearly, $\tau(\tea)\in T_{\A}$; we claim that $\tau(\xx).\beta_1(\xx) = \beta_2(\xx)$ for $\xx\in P\times_{\A}Q$. First consider the case $\xx\in P^i\times Q^i$, with $P^i$ and $Q^i$ defined as in Lemma \ref{lem:subBlendingFctsSumTo1}. By the Orbit-Cone Correspondence applied to the $T_{\A}$-action, all coordinates in $\beta_1(\xx)$ and $\beta_2(\xx)$ vanish except for those graded by $\fata^i$. By Lemma \ref{lem:subBlendingFctsSumTo1}, $\sum_{j\in|\B^i|}\beta^i_j(\xx) = \sum_{k\in|\C^i|}\beta^i_k(\xx) = 1$, so in particular the claim holds. Now consider the case where $\xx\notin \bigcup_{i\in |\A|}P^i\times Q^i$. Then, again by the Orbit-Cone Correspondence applied to the $T_{\A}$-action, for each $i\in|\A|$ there exist $j\in|\B^i|$ and $k\in|\C^i|$ such that $\beta^i_j(\xx),\beta^i_k(\xx) \neq 0$. Thus, by definition, $\tau(\xx).\beta_1(\xx) = \beta_2(\xx)$. We conclude that for all $\xx\in P\times_{\A}Q$, $\beta_1(\xx)$ and $\beta_2(\xx)$ lie in the same $T_{\A}$-orbit. Equality of $\beta_1$ and $\beta_2$ then follows once there exists at least one point in each orbit where the two parametrizations agree. This is indeed the case: for the maximal orbit this is the point given in Lemma \ref{lem:fctsAgreeOnMLE}, for smaller orbits corresponding to faces of $P^i\times Q^i$ we can pick a point as in Lemma \ref{lem:subBlendingFctsSumTo1}. 

It now remains to show that the $\beta^i_{j,k}$ sum to one and that they satisfy linear precision. This follows from direct computation. Firstly, we have
\[\sum_{(i,j,k)\in|\tfp|}\beta^i_{j,k} = \sum_{i\in|\A|, k\in|\C^i|}\beta^i_k\sum_{j\in|\B^i|}\frac{\beta^i_j}{\sum_{j' \in |\B^i|} \beta^{i}_{j'}} = \sum_{i\in|\A|, k\in|\C^i|}\beta^i_k = 1.\]
Finally, we compute
\begin{align*}
& \sum_{(i,j,k)\in|\tfp|}\beta^i_{j,k}(\fatp)\m^i_{j,k} = \sum_{i\in|\A|, j\in|\B^i|}\beta^i_j(\fatp)\sum_{k\in|\C^i|} \frac{\beta^i_k(\fatp)}{\sum_{k' \in |\C^i|} \beta^{i}_{k'}(\fatp)} (\bb^i_j, 0 ) \\
& \hspace{2em} + \sum_{i\in|\A|, k\in|\C^i|}\beta^i_k(\fatp)\sum_{j\in|\B^i|} \frac{\beta^i_j(\fatp)}{\sum_{j' \in |\B^i|} \beta^{i}_{j'}(\fatp)} (0, \cea^i_k)\\
& \hspace{2em} = \left(\sum_{i\in|\A|, j\in|\B^i|}\beta^i_j(\fatp)\bb^i_j~,~ \sum_{i\in|\A|, k\in|\C^i|}\beta^i_k(\fatp)\cea^i_k\right) = \fatp.
\end{align*}
Therefore, the $\beta^i_{j,k}$ constitute blending functions with rational linear precision.
\end{proof}

\section{The Horn matrix of ML degree one toric fiber products}
\label{sec:Horn}
In this section we give an explicit description
of a Horn pair for the toric fiber product of two toric varieties with ML degree one.
This 
construction uses a Horn pair for each factor and for the $(\A-1)$-dimensional probability simplex. Throughout this section we use
notation and setup for the toric fiber product introduced in Section~\ref{sec:TFP}. 
First, we recall the definition of Horn matrix and Horn pair as for example in \cite{DMS21}. Next, in Example~\ref{ex:hornpairs}, we give Horn matrices for the $n
$-dimensional probability simplex, the unit square, and the trapezoid considered in Example~\ref{ex:squareTrapezoidTFP}. Given two vectors $u,v$ with the
same number of entries, we use $u^v$ to denote the product $\prod_i u_i^{v_i}$.

\begin{definition}
    A \emph{Horn matrix} is an $r\times d$ integer matrix with all column sums being zero. Given a Horn matrix $H$ with columns $h_1,h_2,\dots,h_d$ and a vector $\lambda\in\rr^d$, the \emph{Horn parametrization} $\varphi_{(H,\lambda)}\colon\rr^d\rightarrow\rr^d$ is the rational map defined by
    \[u\mapsto \lambda \star (Hu)^H =\left(\lambda_1(Hu)^{h_1},\lambda_2(Hu)^{h_2},\dots,\lambda_d(Hu)^{h_d}\right).\]
\end{definition}

\begin{definition}
The pair $(H,\lambda)$ is called a \emph{Horn pair} if 
\begin{compactenum}
    \item the coordinates of $\varphi_{(H,\lambda)}$ sum up to one, i.e.\ $\lambda_1(Hu)^{h_1}+\lambda_2(Hu)^{h_2}+\dots+\lambda_d(Hu)^{h_d} = 1$, and
    \item $\varphi_{(H,\lambda)}$ is defined for all positive vectors and maps these to positive vectors.
\end{compactenum}
\end{definition}

If $X$ is a statistical model with ML degree one, then, by the results of \cite{Huh14} and \cite{DMS21}, there exist a Horn pair $(H,\lambda)$ such that the MLE $\Phi$ of $X$ satisfies $\Phi=\varphi_{(H,\lambda)}$.
Thus if $X_{\B}$ and $X_{\C}$ have ML degree one, there exist Horn pairs $(H_{\B}, \lambda_{\B})$ and $(H_{\C}, \lambda_{\C})$ such that the maximum likelihood estimate can be expressed as a Horn parametrization, i.e.
\[
   \hat{p}_{\B} = \lambda_{\B} \star (H_{\B}u_{\B})^{H_{\B}} \text{ and } \hat{p}_{\C} = \lambda_{\C} \star (H_{\C}u_{\C})^{H_{\C}}
\]
for data vectors $u_{\B}$ and $u_{\C}$. It follows from \cite[Thm.\ 5.5]{amendola2020maximum} 
that the toric fiber product of the two models 
$X_{\B\times_{\A}\C}$ has again ML degree one and must therefore admit a Horn pair 
$(H_{\B\times_{\A}\C}, \lambda_{\B\times_{\A}\C})$. We will give an explicit description of 
$(H_{\B\times_{\A}\C}, \lambda_{\B\times_{\A}\C})$ in Proposition \ref{thm:HornFiber} below.\par 
To set up the notation, let 
\[u = \left(u^i_{j,k}\right)^{i\in|\A|}_{(j,k)\in |\B^i\times \C^i|}\]
denote a data vector. As before, we will reserve $i,j$ and $k$ for indices of $\A,\B^i$ and $\C^i$, respectively. We use ``$+$'' to denote summation over all possible values of the respective index, e.g.\ $u^i_{j,+} = \sum_{k\in |\C^i|} u^i_{j,k} = \left(u_{\B}\right)^i_j$. In a similar vein, we denote by 
\[p = \left(p^i_{j,k}\right)^{i\in|\A|}_{(j,k)\in |\B^i\times \C^i|}\]
a joint probability distribution for the model $X_{\B\times_{\A}\C}$.

In general, if a statistical model possesses a Horn pair, i.e.\ the Horn parametrization yields a parametrization of the model, the Horn pair is not unique. However, there exists a minimal Horn matrix to a model with ML degree one, see \cite{DMS21}.

\begin{example} \label{ex:hornpairs}
A Horn pair corresponding to the simplex $\Delta_{n}$ is given by letting the Horn matrix be the 
$(n+1)\times (n+1)$-identity matrix with an additional row of $(-1)$s at the bottom and with $\lambda$ being the vector of all $(-1)$s. For the one-dimensional simplex $\Delta_1$ we have
\[H=\begin{pmatrix}
    1 & 0\\
    0 & 1\\
    \minus 1 & \minus 1\\
\end{pmatrix},~\lambda=(-1,-1),~\Phi(u_1,u_2)=\lambda\star (Hu)^H=\left(\frac{u_1}{u_1+u_2},\frac{u_2}{u_1+u_2}\right).\]
For another illustration, consider the two models $X_{\B}$ and $X_{\C}$ defined by the polytopes $P=\conv(\B)$ and $Q=\conv(\C)$ from Example \ref{ex:tbfs}. Note that $X_{\B}$ is the well-known independence model of
two binary random variables, and $X_{\C}$ is a multinomial staged tree. For toric surfaces with ML degree one, the Horn pair can be directly read off from the lattice distance functions and the normal fan of the polytope, see \cite[Prop.\ 3.1]{familiesEliana}. Concretely, we have
\[H_{\B} = \begin{pNiceMatrix}[first-row]
    \bb^1_1 & \bb^1_2 & \bb^2_1 & \bb^2_2\\
    1 & 0 & 1 & 0\\
    0 & 1 & 0 & 1\\
    1 & 1 & 0 & 0\\
    0 & 0 & 1 & 1\\
    \minus 1 & \minus 1 & \minus 1 & \minus 1\\
    \minus 1 & \minus 1 & \minus 1 & \minus 1\\
\end{pNiceMatrix}, ~H_{\C} = \begin{pNiceMatrix}[first-row]
    \cea^1_1 & \cea^1_2 & \cea^1_3 & \cea^2_1 & \cea^2_2\\
    0 & 1 & 2 & 1 & 0\\
    0 & 0 & 0 & 1 & 1\\
    2 & 1 & 0 & 0 & 1\\
    1 & 1 & 1 & 0 & 0\\
    \minus 1 & \minus 1 & \minus 1 & \minus 1 & \minus 1\\
    \minus 2 & \minus 2 & \minus 2 & \minus 1 & \minus 1\\
\end{pNiceMatrix}\]
and $\lambda_{\B}=(1,1,1,1),~\lambda_{\C}=(-1,-2,-1,1,1)$; the columns of the Horn matrices are labelled by the vectors of $\B$ and $\C$, respectively.
\end{example}

\begin{proposition}
\label{thm:HornFiber}
Let $X_{\B}$ and $X_{\C}$ be toric varieties with ML degree one and correspondng 
Horn pairs $(H_{\B},\lambda_{\B})$ and $ (H_{\C},\lambda_{\C})$, respectively, where  $H_{\B} \in \zz^{r_1\times |\B|}, H_{\C}\in \zz^{r_2\times |\C|}$. Fix $H_{\A}$ to be the minimal Horn matrix
associated to the $(|\A|-1)$-dimensional probability simplex, so $H_{\A}\in \zz^{(|\A|+1)\times |\A|}$. Denote the columns of $H_{\B},H_{\C}$, and $H_{\A}$ by $h_{j}^i,h_{k}^i$, and $h^i,$ respectively. Then $(H_{\tfp},\lambda_{\tfp})$ is a Horn pair for the 
toric fiber product $X_{\tfp}$.
Here, the vector $\lambda_{\tfp}$ of coefficients is given by
\[
\lambda_{\B\times_{\A}\C} = 
\left(\lambda^i_{{j,k}}\right)^{i\in |\A|}_{(j,k)\in |\B^i\times \C^i|}
\text{ with } 
\lambda^i_{{j,k}} = -\lambda^i_{{j}}\lambda^i_{{k}} \text{ and }
\lambda_{\B}=(\lambda_j^i)_{j\in |\B^i|}^{i\in |A|}, \lambda_{\C}=
(\lambda_k^i)_{k\in |\C^i|}^{i\in |A|}, \]
and the Horn matrix $H_{\tfp}$ is given in block form by
\begin{align}
    H_{\B\times_{\A}\C} = 
    \left(
    H_{\B^1\times \C^1}\mid H_{\B^2\times \C^2}
    \mid \cdots \mid H_{\B^{|\A|}\times \C^{|\A|}}
    \right).
\end{align}
For each $i\in |\A|$, the column $h_{j,k}^i$,  of block  $H_{\B^i\times \C^i}$ 
is the vertical concatenation of $h_j^i,h_k^i,-h^i$. 
Explicitly, if $\rho=r_1+r_2+|\A|+1$ and
$\alpha \in [\rho]$, then the row $\alpha$ of $h_{j,k}^i$, denoted by $h_{j,k}^{\alpha,i}$,
is given by
\[h^{\alpha,i}_{{j,k}} = \left\{
\begin{array}{ll}
  h^{\alpha,i}_{{j}}   &  \text{ for } 1\leq \alpha\leq r_1\\
  h^{(\alpha-r_1),i}_{{k}}   &  \text{ for } r_1+1\leq \alpha\leq r_1+r_2\\
  -h^{(\alpha-r_1-r_2),i}   &  \text{ for } r_1+r_2+1\leq \alpha\leq \rho.\\
\end{array}
\right.\]
Where,  $h^{\alpha,i}_{{j}},h^{(\alpha-r_1),i}_{{k}}$, and $h^{(\alpha-r_1-r_2),i}$, are the entries 
$\alpha,\alpha-r_1$, and $\alpha-r_1-r_2$ of the columns $h_j^i,h_k^i,h^i$, respectively.
\end{proposition}

\begin{proof}
It suffices to check that the pair $(H_{\B\times_{\A}\C}, \lambda_{\B\times_{\A}\C})$ gives rise to a Horn parametrization yielding the correct expression for the maximum likelihood estimate of $X_{\B\times_{\A}\C}$; then the pair will automatically be friendly and positive and thus a Horn pair for $X_{\B\times_{\A}\C}$, see \cite{DMS21}.\par 
By \cite[Thm.\ 5.5]{amendola2020maximum}, the MLE of $X_{\B\times_{\A}\C}$ is given by
$$\hat{p} = \left(\hat{p}^i_{j,k}\right)^{i\in |\A|}_{(j,k)\in |\B^i\times \C^i|} \quad\text{with}\quad \hat{p}^i_{j,k} = \frac{\hat{p}^i_{j} \hat{p}^i_{k}}{\hat{p}^i_{+,+}}.$$
The $(i,j,k)$th entry of the Horn parametrization computes as
\begin{align*}
    &\left(\lambda_{\B\times_{\A}\C}\star \left(H_{\B\times_{\A}\C} u\right)^{H_{\B\times_{\A}\C}}\right)^i_{j,k} = \lambda^i_{{j,k}} \prod_{\alpha =1}^{\rho} \left( \sum_{(\ti,\tj,\tk)\in |\B\times_{\A}\C|} h^{\alpha,\ti}_{{\tj,\tk}} u^{\ti}_{\tj,\tk}\right)^{h^{\alpha,i}_{{j,k}}} \stepcounter{equation}\tag{\theequation}\label{equ:HornUniformisation} \\
\end{align*}
Let us split the product above into three products $P_1,P_2$ and $P_3$ where $\alpha$ ranges over $\{1,\dots,r_1\}$, $\{r_1+1,\dots,r_1+r_2\}$ and $\{r_1+r_2+1,\dots, \rho\}$, respectively. Then we obtain
\begin{align*}
    P_1 & = \prod_{\alpha =1}^{r_1} \left( \sum_{(\ti,\tj,\tk) \in |\B\times_{\A}\C|} h^{\alpha,\ti}_{{\tj,\tk}} u^{\ti}_{\tj,\tk}\right)^{h^{\alpha,i}_{{\B\times_{\A}\C}_{j,k}}} = \prod_{\alpha =1}^{r_1} \left( \sum_{(\ti,\tj,\tk) \in |\B\times_{\A}\C|} h^{\alpha,\ti}_{{\tj}} u^{\ti}_{\tj,\tk}\right)^{h^{\alpha,i}_{{j}}}\\
    & = \prod_{\alpha =1}^{r_1} \left( \sum_{(\ti,\tj) \in |\B|} h^{\alpha,\ti}_{{\tj}} u^{\ti}_{{\tj},+}\right)^{h^{\alpha,i}_{{\B}_{j}}} = \frac{\hat{p}^i_{j,+}}{\lambda^i_{j}},
\end{align*}
and similarly $P_2 = \frac{\hat{p}^i_{+,k}}{\lambda^i_{k}}$.
Finally, we have
\begin{align*}
    P_3 & = \prod_{\alpha =r_1+r_2+1}^{\rho} \left( \sum_{(\ti,\tj,\tk) \in |\B\times_{\A}\C|} h^{\alpha,\ti}_{{\tj,\tk}} u^{\ti}_{\tj,\tk}\right)^{h^{\alpha,i}_{{j,k}}} 
    = \prod_{\alpha =r_1+r_2+1}^{\rho} \left( \sum_{\ti \in |\A|} -h^{(\alpha-r_1-r_2),\ti} u^{\ti}_{+,+}\right)^{\left(-h^{\alpha,i}\right)}\\
    &= \left(\prod_{\alpha = 1}^{|\A|} \left(-u^{\alpha}_{+,+}\right)^{\delta_{\alpha,i}}\right)^{(-1)}\cdot u^+_{+,+} = -\frac{u^+_{+,+}}{u^i_{+,+}}.
\end{align*}
As $\A$ is linearly independent, 
$\hat{p}^i_{+,+} = \frac{u^i_{+,+}}{u^+_{+,+}}.$
Combining this with the computations above, we obtain
\[(\ref{equ:HornUniformisation}) = -\lambda^i_{j}\lambda^i_{k} P_1P_2P_3 = \frac{\hat{p}^i_{j,+} \hat{p}^i_{+,k}}{\hat{p}^i_{+,+}}.\]\end{proof}

\begin{example}
The Horn pair for the toric fiber product $X_{\B}\times_{\A}X_{\C}$ from Proposition~\ref{thm:HornFiber}, where $X_{\B}$ and $X_{\C}$
are defined in  Example \ref{ex:hornpairs} and the multigrading is specified 
in Figure \ref{fig:squareTrapezoidTFP}, is given by 
\[H_{\tfp}=\begin{pNiceMatrix}[columns-width = .05cm, first-row]
\begin{array}{cc} \bb_1^1  \\ \cea_1^1 \end{array} & \begin{array}{cc} \bb_1^1  \\ \cea_2^1 \end{array} & \begin{array}{cc} \bb_1^1  \\ \cea_3^1 \end{array} & \begin{array}{cc} \bb_2^1  \\ \cea_1^1 \end{array} & \begin{array}{cc} \bb_2^1  \\ \cea_2^1 \end{array} & \begin{array}{cc} \bb_2^1  \\ \cea_3^1 \end{array} & \begin{array}{cc} \bb_1^2  \\ \cea_1^2 \end{array} & \begin{array}{cc} \bb_1^2  \\ \cea_2^2 \end{array} & \begin{array}{cc} \bb_2^2  \\ \cea_1^2 \end{array} & \begin{array}{cc} \bb_2^2  \\ \cea_2^2 \end{array} \\
 1  &  1  &  1  &  0  &  0  &  0  &  1  &  1  &  0  &  0  \\
 0  &  0  &  0  &  1  &  1  &  1  &  0  &  0  &  1  &  1  \\
 1  &  1  &  1  &  1  &  1  &  1  &  0  &  0  &  0  &  0  \\
 0  &  0  &  0  &  0  &  0  &  0  &  1  &  1  &  1  &  1  \\
 \minus 1 & \minus 1 & \minus 1 & \minus 1 & \minus 1 & \minus 1 & \minus 1 & \minus 1 & \minus 1 & \minus 1\\
 \minus 1 & \minus 1 & \minus 1 & \minus 1 & \minus 1 & \minus 1 & \minus 1 & \minus 1 & \minus 1 & \minus 1\\
 0 & 1 & 2 & 0 & 1 & 2 & 1 & 0 & 1 & 0\\
 0 & 0 & 0 & 0 & 0 & 0 & 1 & 1 & 1 & 1\\
 2 & 1 & 0 & 2 & 1 & 0 & 0 & 1 & 0 & 1\\
 1 & 1 & 1 & 1 & 1 & 1 & 0 & 0 & 0 & 0\\
 \minus 1 & \minus 1 & \minus 1 & \minus 1 & \minus 1 & \minus 1 & \minus 1 & \minus 1 & \minus 1 & \minus 1\\
 \minus 2 & \minus 2 & \minus 2 & \minus 2 & \minus 2 & \minus 2 & \minus 1 & \minus 1 & \minus 1 & \minus 1\\
 \minus 1 & \minus 1 & \minus 1 & \minus 1 & \minus 1 & \minus 1 & 0 & 0 & 0 & 0\\
 0 & 0 & 0 & 0 & 0 & 0 & \minus 1 & \minus 1 & \minus 1 & \minus 1\\
 1 & 1 & 1 & 1 & 1 & 1 & 1 & 1 & 1 & 1\\
\end{pNiceMatrix}\]
and $\lambda_{\tfp} = (1, 2, 1, 1, 2, 1, \minus 1, \minus 1, \minus 1, \minus 1)$. Note that in almost all instances, the Horn matrix as constructed in Proposition \ref{thm:HornFiber} will not be minimal, as is also the case in this example. However, it can be transformed into a minimal one via an efficient algorithm \cite[Lem.\ 3]{DMS21}.
\end{example}

{\bf Acknowledgements}. 
Eliana Duarte was supported by the FCT grant 2020.01933.CEECIND, and partially supported by CMUP under the FCT grant UIDB/00144/2020.

\bibliographystyle{splncs04}
\bibliography{bibliography}

\end{document}